\documentclass[10pt,reqno]{amsart}

\usepackage{amsmath, amsfonts, amssymb, latexsym, amsthm}

\usepackage{amsmath, amsfonts, amssymb, latexsym}

\usepackage{mathrsfs}

\newtheorem{theorem}{Theorem}
\theoremstyle{plain}

\newtheorem{lemma}[theorem]{Lemma}
\newtheorem{definition}[theorem]{Definition}

\newtheorem{proposition}[theorem]{Proposition}
\newtheorem{corollary}[theorem]{Corollary}
\newtheorem{remark}[theorem]{Remark}

\numberwithin{equation}{section}
\numberwithin{theorem}{section}

\newcommand{\cB}{\mathcal{B}}

\newcommand{\cE}{\mathcal{E}}
\newcommand{\cK}{\mathcal{K}}

\newcommand{\cJ}{\mathcal{J}}

\newcommand{\cF}{\mathcal{F}}

\newcommand{\cT}{\mathcal{T}}
\newcommand{\cW}{\mathcal{W}}

\newcommand{\E}{\mathbb{E}}
\newcommand{\R}{\mathbb{R}}
\newcommand{\F}{\mathbb{F}}
\newcommand{\N}{\mathbb{N}}
\newcommand{\bP}{\mathbb{P}} 

\def\e{\varepsilon}
\def\ol{\overline}

\def\Ito{It\^{o} }

\begin{document}
\title
{Observability Inequality of Backward Stochastic Heat Equations for Measurable
  Sets and Its Applications}

\author{Donghui Yang}
\author{Jie Zhong}
\address{Donghui Yang: School of Mathematics and Statistics,
    Central South University, Changsha, 410075, P.R. China}
\email{donghyang@139.com}
\thanks{The first author was supported by the
    National Science Foundation of China and the grant of China
    Scholarship Council.}
\address{Jie Zhong: Department of Mathematics,
    University of Central Florida, Orlando, 32826, USA} 
\email{jiezhongmath@gmail.com}


\maketitle

\begin{abstract}
This paper aims to provide directly the observability inequality of
 backward stochastic heat equations for measurable sets. As an
 immediate application, the null controllability of the forward heat equations is
 obtained. Moreover, an
 interesting relaxed optimal actuator location problem is formulated, and the existence of its solution is
 proved. Finally, the solution is characterized by
 a Nash equilibrium of the associated game problem.
\end{abstract}

\pagestyle{myheadings}
\thispagestyle{plain}
\markboth{OBSERVABILITY INEQUALITY AND ITS APPLICATIONS}{DONGHUI YANG AND JIE ZHONG}

\section{Introduction}
Observability inequality is an important and powerful tool for the study of
stabilization and controllability problems of partial differential
equations. However, most of related works for heat equations concern with the internal
control living on an open subset. Recently, the authors in
\cite{Apraiz2014, PW13} establish the observability inequality of the heat
equation for the measurable subsets, and show the null controllability
with controls restricted over these sets. This generalization
facilitates the study of the optimal actuator location problem for
a wider class of equations. For example, compared to the one
dimensional case studied in \cite{Allaire2010} and a special class of
controlled domains considered in \cite{Guo2014}, the authors in
\cite{Guo2015} investigate the optimal actuator location of the
minimum norm controls for heat equations in arbitrary dimensions,
and the actuator domain is only required to have a prescribed Lebesgue
measure.

One of the main
contributions of this paper is the direct derivation of the observability inequality for
stochastic backward heat equations for measurable subsets,  which is considered very challenging and difficult 
in \cite[page 99
and page 108-110]{Viorel2003}. By duality we obtain the
null controllability for the corresponding forward equation. Our
results extend the deterministic case to the stochastic
counterpart. It is worth noting that we cannot simply mimic the
calculations in the deterministic case by applying the time change technique, and
treat the backward and forward equations in the same way, since
adaptedness is always required in the stochastic system. On the other
hand, our observability estimate also recovers the result in
\cite[Proposition 4.1]{Lu2011}, where only open controlled domain is
considered, and the result is obtained by null controllability. For more general stochastic parabolic equations,
but with two controls, we refer the reader to the work in \cite{Tang2009}.

As an important application, we consider the optimal actuator location of
the minimum norm control
problem for internal null controllable stochastic heat equations. In
fact, the actuator location problem for deterministic equations has
been widely studied; see for example,
\cite{Allaire2010,Darivandi2013,Guo2015,Privat2014}, and also
numerical research in \cite{Munch2008,Munch2009,Tiba2011}. To
the best of our knowledge, this paper is the first attempt to consider the
shape optimization for the stochastic system. We
show the existence of the minimum norm control, which is done by
solving a variational problem with suitable norms guaranteed by the
observability inequality. Then we prove the existence of the relaxed
optimal actuator location and characterize the solution of the relaxed
problem via a Nash equilibrium.

Before we state our main theorems, let us introduce necessary notations.

Let 
$T>0$ be a fixed positive time constant, 
and $D$ be a bounded domain in 
$\R^d$ with a $C^2$ boundary 
$\partial D$. Let 
$E$ and 
$G$ be measurable subsets with positive measures of 
$[0,T]$ and 
$D$, respectively. 

Throughout this paper, we denote by
$(\cdot, \cdot)$ the inner product in $L^2(D)$, and denote by
$\|\cdot\|$ the norm induced by $(\cdot,\cdot)$. We also use the
notations $(\cdot,\cdot)_G$ and $\|\cdot\|_G$ for the inner product
and the norm defined on $L^2(G)$, respectively. We denote by $|\cdot|$
the Lebesgue measure on $\R^d$.

Let $\F = (\Omega, \cF,\{\cF_t\}_{t\ge 0}, \bP)$ be a stochastic
basis with usual conditions. On $\F$, we define a standard scalar Wiener
process $W = \{w(t)\}_{t\ge 0}$. For simplicity, we assume that the
filtration 
$\{\cF_t\}_{t\ge0}$ is generated by $W$.

Given a Hilbert space $H$, we denote by $L^2_\cF(0,T;H)$ the Banach
space consisting of all $H$-valued $\{\cF_t\}_{t\ge0}$-adapted
processes $X$ such that the square of the canonical norm  $\E\|X(\cdot)\|^2_{L^2(0,T;H)}
<\infty$; denote by
$L^\infty_\cF(0,T;H)$ the Banach space consisting of all $H$-valued
$\{\cF_t\}_{t\ge0}$-adapted bounded processes, with the essential
supremum norm; and denote by $L^2_\cF(\Omega;C([0,T];H))$ the Banach
space consisting of all $H$-valued $\{\cF_t\}_{t\ge0}$-adapted continuous
processes $X$ such that the square of the canonical norm $\E
\|X(\cdot)\|^2_{C(0,T;H)}<\infty$. For any $t\in[0,T]$, the space
$L^2(\Omega,\cF_t,\bP;H)$ consists of all $H$-valued
$\cF_t$-measurable random variables with finite second moments.

Let $A$ be an unbounded linear operator on $L^2(D)$:
\[\mathcal{D}(A)=H^2(D)\cap H^1_0(D),\quad Av = \Delta v,\ \forall v\in
\mathcal{D}(A).\]

The goal of this paper is to derive directly the observability
inequality for the following backward stochastic heat equation
\begin{equation}
  \label{eq:adjoint}
  \begin{cases}
dz = - A z dt - a(t) Z dt + Z dw(t),
& t\in (0,T),\\
z(T) = \eta,
\end{cases}
\end{equation}
where $a\in L^\infty_\cF(0,T;\R)$. For each $\eta\in
L^2(\Omega,\cF_T,\bP;L^2(D))$, it is known (see for example
\cite{Hu1991,Du2010}) that the equation \eqref{eq:adjoint} admits a
unique solution $(z,Z)$ in the space of
$(L^2_\cF(\Omega;C(0,T;L^2(D)))\cap L^2_\cF(0,T;H^1_0(D)))\times L^2_\cF(0,T;L^2(D))$.

The following is our main theorem.

\begin{theorem}
  \label{thm:obs-inq}
Let $D$ be a bounded domain in $\R^d$ with a $C^2$ boundary. Let
  $x_0\in D$ and $R\in (0,1]$ such that $B_{4R}(x_0)\subseteq
  D$. Suppose $G$ is a subset of $D$ with positive measure, contained
  in $B_R(x_0)$, and $E$ is measurable subset of $[0,T]$ with positive
  measure. Then there exists a constant $C=C(T,D,R,|G|, |E|)$ such that the
  following observability inequality holds: for any $\eta\in L^2(\Omega,\cF_T,\bP;L^2(D))$,
\begin{equation}
  \label{eq:obs-inq}
  \E \|z(0;T,\eta)\|^2 \le C \left(\int_E
    \left(\E\|z(t;T,\eta)\|^2_G\right)^{1/2}dt \right)^2.
\end{equation}
\end{theorem}

As a result, we obtain the null controllability for a class of forward
stochastic heat equations:
\begin{equation}
  \label{eq:main}
  \begin{cases}
d y = A ydt + \chi_E
\chi_G u(t) dt + a(t) y dw(t),
& t\in (0,T),\\
y(0) = y_0.
\end{cases}
\end{equation}

\begin{theorem}
  \label{thm:main}
The equation \eqref{eq:main} is $L^\infty$-null controllable. That is,
for each initial data $y_0\in L^2(\Omega, \cF_0, \bP;L^2(D))$, there
is a control $u$ in the space $L^\infty_\cF (0,T;L^2(D))$ such that
the solution $y$ of the equation \eqref{eq:main} satisfies $y_T = 0$
in $D$, $\bP$-a.s. Moreover, the control $u$ satisfies the following
estimate
\begin{equation}
  \label{eq:control est.}
  \E\|u\|^2_{L^\infty (0,T;L^2(D))} \le C \E \|y_0\|^2.
\end{equation}
\end{theorem}

The rest of the paper is organized as follows. In Section
\ref{sec:obv-null}, we prove our main theorems. In Section
\ref{sec:relaxed}, we discuss the relaxed optimal actuator location
problem. More specifically, we state and formulate the problem in
Section \ref{subsec:formulation}. In Section 3.2, we show the
existence of the optimal minimal norm control. In Section
\ref{ss:relaxed optimal}, the existence of relaxed optimal actuator
location is proved. Finally, Section \ref{sec:char-via-nash} provides
the characterization of the solution of the relaxed optimal actuator
location problem by a Nash equilibrium. For completeness, we include
some basics of two person zero sum game in Appendix.

\section{Observability Inequality and Null Controllability}
\label{sec:obv-null}
In this section, we will prove our main therorem and provide the
observability inequality \eqref{eq:obs-inq}. By duality, the
equivalence between the null controllability of the equation
\eqref{eq:main} and the observability estimate for the adjoint
equation \eqref{eq:adjoint} is obtained. As a result, we obtain Theorem \ref{thm:main}.

Let us start with some notations. we write 
\[
0<\lambda_1\le \lambda_2\le \cdots
\]
for the eigenvalues of $-\Delta$ with the zero Dirichlet boundary
condition over $\partial D$, and $\{e_j\}_{j\ge 1}$ for the
orthonormal basis for $L^2(D)$. For each $\lambda>0$, we define
\[
\cE_\lambda f = \sum_{\lambda_j\le\lambda} (f,e_j)e_j,~\text{and}~
\cE^\perp_\lambda f = \sum_{\lambda_j>\lambda} (f,e_j)e_j.
\]
 
Now recall an important spectral inequality used later in this paper; see Theorem 5 in \cite{Apraiz2014}.
\begin{lemma}
  \label{lem:spectral}
  Let $D$ be a bounded domain in $\R^d$ with a $C^2$ boundary. Let
  $x_0\in D$ and $R\in (0,1]$ such that $B_{4R}(x_0)\subseteq
  D$. Suppose $G$ is a subset of $D$ with positive measure, contained
  in $B_R(x_0)$. Then there exists a positive constant $N=N(D,R,|G|)$ such that
\begin{equation}
  \label{eq:spectral}
\|\cE_\lambda \eta\|^2
\leq N\exp\big( N\sqrt{\lambda}\big)\|\cE_\lambda \eta \|_G^2,\  \forall \eta\in
L^2(D), \ \lambda>0.
\end{equation}
\end{lemma}

Set $\tau = \|a\|^2_{L^\infty_\cF(0,T;\R)}$.

Let us denote by $z(\cdot;T,\eta)$ the solution of equation
\eqref{eq:adjoint} given the terminal condition $\eta=z(T)$. By
linearity, it is easy to check that
\begin{align}
  & z(t; T,\cE_\lambda \eta) = \sum_{\lambda_j\le \lambda}
  z_j(t;T,\eta_j)e_j = \cE_\lambda z(t;T,\eta);\label{eq:1}\\
& z(t; T,\cE^\bot_\lambda \eta) = \sum_{\lambda_j> \lambda}
  z_j(t;T, \eta_j)e_j = \cE^\bot_\lambda z(t;T,\eta), \label{eq:2}
\end{align}
where $\eta_j=(\eta,e_j)$ and $(z_j(\cdot;T,\eta_j), Z_j(\cdot;T,\eta_j))$ is the solution of the following backward
stochastic differential equation
\begin{equation}\label{eq:bsde}
\begin{cases}
dz_j = \lambda_j z_j dt - a(t) Z_j dt + Z_j dW(t),
& t\in (0,T),\\
z_j(T) = \eta_j.
\end{cases}
\end{equation}

\begin{lemma}
      Given any $\eta$ in the space of  $L^2(\Omega,\cF_T,\bP;L^2(D))$,
    we have for each $t\in [0,T]$,
  \begin{equation}
    \label{eq:decay}
    \E \|z(t ;T,\cE_\lambda^\bot \eta)\|^2 \le e^{(-2\lambda+\tau)(T-t)} \E \|\eta\|^2
  \end{equation}
\end{lemma}
\begin{proof}
  Applying \Ito formula to
  $\exp[(2\lambda-\tau)(T-t)]\|z(t;T,\cE_\lambda^\bot)\|^2$, we obtain
\begin{align*}
  & \|\cE^\bot_\lambda \eta\|^2 - e^{(2\lambda-\tau)(T-t)}
    \|z(t;T,\cE_\lambda^\bot\eta)\|^2\\
= &\ \int_t^T e^{(2\lambda-\tau)(T-s)} \left[2(z, Az) -
    2a(s)(z, Z)\right] ds\\
& \qquad + \int_t^T e^{(2\lambda-\tau)(T-s)} \|Z\|^2 ds + \int_t^T e^{(2\lambda-\tau)(T-s)}
    2(z, Z) dW(s)\\
&\qquad  - \int_t^T e^{(2\lambda-\tau)(T-s)}
    (2\lambda-\tau) \|z(s;T,\cE_\lambda^\bot\eta)\|^2 ds.
\end{align*}
Taking the expectation, it follows from equality \eqref{eq:2} that
\begin{align*}
  & \E \|\cE^\bot_\lambda \eta\|^2 - e^{(2\lambda-\tau)(T-t)}
    \E \|z(t;T,\cE_\lambda^\bot \eta)\|^2\\
= &\ \E \int_t^T e^{(2\lambda-\tau)(T-s)} \Big( 2
  \sum_{\lambda_j>\lambda} \lambda_j (z^j)^2 -2a(s)(z,Z)\\
&\qquad \left.+ \|Z\|^2- (2\lambda-\tau) \|z(s;T,\cE_\lambda^\bot\eta)\|^2\right) dt\\
\ge & \ \E \int_t^T e^{(2\lambda-\tau)(T-s)} \Big(-2a(s)(z, Z)
  \\
& \qquad + \left.
  \|Z\|^2 + \tau \|z(s;T,\cE_\lambda^\bot\eta)\|^2\right) dt\\
\ge & \ \E \int_t^T e^{(2\lambda-\tau)(T-s)}
\Big((\|Z\|-|a(s)|\|z(s;T,\cE_\lambda^\bot\eta)\|)^2\\
& \qquad + (\tau-|a(s)|^2)\|z(s;T,\cE_\lambda^\bot\eta)\|^2\Big) dt\\
\ge &\ 0,
\end{align*}
which implies the inequality \eqref{eq:decay}.
\end{proof}

Next, we provide an interpolation inequality.

\begin{proposition}
  For $\eta\in L^2(\Omega,\cF_T,\bP;L^2(D))$, and $t\in [0,T)$, there
  exists
 a constant $K = K(T,D,R,|G|)$ such that
\begin{equation}
  \label{eq:interpolation}
  \E \|z(t;T,\eta)\|^2 \le K\exp(K(T-t)^{-1}) (\E \|z(t;T,\eta)\|_G^2)^{1/2}(\E \|\eta\|^2)^{1/2}.
\end{equation}
\end{proposition}

\begin{proof}
  Set $z = z(\cdot;T,\eta)$, then it follows from the
  spectral estimate
  \eqref{eq:spectral} that
\begin{align*}
  \E \|\cE_\lambda z(t)\|^2 
& \le N \exp(N\sqrt{\lambda})\E \|\cE_\lambda z(t)\|^2_G\\
& \le N \exp(N\sqrt{\lambda})\big(\E \|z(t)\|^2_G +
\E\|\cE^\bot_\lambda z(t)\|^2_G\big)
\end{align*}
for some constant $N = N(D,R,|G|)$. Therefore, by the decay estimate
\eqref{eq:decay} we obtain that
\begin{align*}
&  \E \|z(t)\|^2 
 = \E \|\cE_\lambda z(t)\|^2 + \E \|\cE^\bot_\lambda
z(t)\|^2\\
& \le  N \exp(N\sqrt{\lambda})\big(\E \| z(t)\|^2_G + \E
\|\cE^\bot_\lambda z(t)\|^2_G\big)+ \E
\|\cE^\bot_\lambda z(t)\|^2\\
& \le 2N \exp(N\sqrt{\lambda})\big(\E \| z(t)\|^2_G + \E
\|\cE^\bot_\lambda z(t)\|^2\big)\\
&\le 2N \exp(N\sqrt{\lambda})\big(\E \| z(t)\|^2_G +
e^{(-2\lambda+\tau)(T-t)} \E \|\eta\|^2\big)\\
&\le  2N e^{\tau T}\exp(N\sqrt{\lambda})\big(\E \| z(t)\|^2_G +
e^{(-2\lambda(T-t))} \E \|\eta\|^2\big)\\
& = 2N e^{\tau T} \exp(N\sqrt{\lambda}-\lambda(T-t))\big(e^{\lambda(T-t)}\E \| z(t)\|^2_G +
e^{-\lambda(T-t)} \E \|\eta\|^2\big).
\end{align*}
It is easy to verify that for all $\lambda>0$, 
\[
 N\sqrt{\lambda} - \lambda(T-t) \le \frac{N^2}{4(T-t)}.
\]
Hence, there exists a constant $K=K(T,D,R,|G|)$ such that 
\[
 \E \|z(t)\|^2 \le K \exp(K(T-t)^{-1})\big[e^{\lambda(T-t)}\E \| z(t)\|^2_G +
e^{-\lambda(T-t)} \E \|\eta\|^2\big],
\]
which is equivalent to 
\begin{equation}
  \label{eq:epsilon inq}
  \E \|z(t)\|^2 \le K \exp(K(T-t)^{-1})\big[\e^{-1}\E \| z(t)\|^2_G +
\e \E \|\eta\|^2\big],\quad \forall \e\in (0,1).
\end{equation}
Noting that $\E \|z(t)\|^2 \le C \E \|\eta\|^2$, where $C$ is a
constant depending on $T$, we see that the inequality
\eqref{eq:epsilon inq} holds for all $\e>0$. Finally, minimizing
\eqref{eq:epsilon inq} with respect to $\e$ leads to the desired
estimate \eqref{eq:interpolation}.
\end{proof}

We are now ready to prove Theorem \ref{thm:obs-inq}
\begin{proof}[Proof of Theorem \ref{thm:obs-inq}]
  Let $\ell\in(0,T)$ be any Lebesgue point of $E$. Then for each constant $q\in(0,1)$ which is to be fixed later, there exists a monotone increasing sequence
$\{\ell_m\}_{m\geq1}$ in $(\ell,T)$ such that
$$\lim_{m\rightarrow +\infty}\ell_m=\ell,$$
\begin{equation}\label{eq:equiv ratio}
\ell_{m+2}-\ell_{m+1}=q(\ell_{m+1}-\ell_m),\;\forall m\geq1
\end{equation}
and
$$|E\cap (\ell_m,\ell_{m+1})|\geq \frac{\ell_{m+1}-\ell_m}{3},\;\forall m\geq1.$$
Set
$$\tau_m=\ell_{m+1}-\frac{\ell_{m+1}-\ell_m}{6},\;\forall m\geq1.$$
For each $t\in (\ell_m, \tau_m)$, by the interpolation inequality
\eqref{eq:interpolation}, we have
\[
\E \|z(t)\|^2 \le K\exp(K(\ell_{m+1}-t)^{-1}) (\E \|z(t)\|_G^2)^{1/2}(\E \|z(\ell_{m+1})\|^2)^{1/2}.
\]
Since
\[
\ell_{m+1} -t \ge \ell_{m+1} -\tau_m = \frac{\ell_{m+1} -\ell_m}{6},
\]
and for some constant $C=C(T)$, $\E \|z(\ell_m)\|^2 \le C \E
\|z(t)\|^2$, there exists a constant $C = C(T,D,R, |G|)$ such
that for all $m\ge 1$, and $t\in (\ell_m, \tau_m)$,
\[
\E \|z(\ell_m)\|^2 \le C e^{\frac{C}{\ell_{m+1}-\ell_m}} (\E \|z(t)\|_G^2)^{1/2}(\E \|z(\ell_{m+1})\|^2)^{1/2},
\]
which implies for each $\e>0$,
\[
\E \|z(\ell_m)\|^2 \le \e^{-1} C
e^{\frac{C}{\ell_{m+1}-\ell_{m}}} \E \|z(t)\|_G^2+ \e \E \|z(\ell_{m+1})\|^2,
\]
by the Cauchy inequality with $\e$. Equivalently, we have
\begin{equation}
  \label{eq:Am-inq}
  A_m \le \e^{-1} C
e^{\frac{C}{\ell_{m+1}-\ell_{m}}} B(t) + \e A_{m+1},
\end{equation}
where
\begin{equation}
  \label{eq:Am-B}
  A_m =\left(\E \|z(\ell_m)\|^2\right)^{1/2},\ B(t)=\left(\E \|z(t)\|^2\right)^{1/2}.
\end{equation}

Integrating the previous
inequality \eqref{eq:Am-inq} over $E\cap (\ell_m,\tau_m)$, and noting that
\begin{align*}
  |E\cap(\ell_m,\tau_m)|
& = |E \cap (\ell_m,\ell_{m+1})| - |E\cap (\tau_m,\ell_{m+1})|\\
& \ge \frac{\ell_{m+1}-\ell_{m}}{3} - \frac{\ell_{m+1}-\ell_{m}}{6}\\
& = \frac{\ell_{m+1}-\ell_{m}}{6},
\end{align*}
we have that for each $\e>0$
\[
A_m \le \e A_{m+1}+ \e^{-1} C
e^{\frac{C}{\ell_{m+1}-\ell_m}} \int_{\ell_{m}}^{\ell_{m+1} } \chi_E B(t) dt.
\]
Multiplying the above inequality by $\e \exp(-C/(\ell_{m+1}-\ell_m))$,
and replacing $\e$ by $\sqrt{\e}$ lead to
\[
\sqrt{\e} e^{-\frac{C}{\ell_{m+1}-\ell_m}} A_m \le \e
e^{-\frac{C}{\ell_{m+1}-\ell_m}} A_{m+1} + C \int_{\ell_{m}}^{\ell_{m+1} } \chi_E B(t) dt.
\]
 Finally choosing $\e =
\exp(-1/(\ell_m-\ell_{m+1}))$ in the above inequality, we get
\begin{align*}
  & e^{-\frac{C+1/2}{\ell_{m+1}-\ell_{m}}} A_m -
  e^{-\frac{C+1}{\ell_{m+1}-\ell_{m}}} A_{m+1}\le C \int_{\ell_{m}}^{\ell_{m+1} } \chi_E B(t) dt.
\end{align*}
Now, choosing $q = \frac{C+1/2}{C+1}$ in \eqref{eq:equiv ratio}, we
have
\begin{align*}
  & e^{-\frac{C+1/2}{\ell_{m+1}-\ell_{m}}} A_m -
  e^{-\frac{C+1/2}{\ell_{m+2}-\ell_{m+1}}} A_{m+1} \le C \int_{\ell_{m}}^{\ell_{m+1} } \chi_E B(t) dt.
\end{align*}
Summing the above inequality from $m=1$ to $+\infty$, we have
\[
A_1 \le C e^{\frac{C+1/2}{\ell_2-\ell_1}}
\int_\ell^{\ell_1} \chi_E B(t) dt.
\]
By the substitution \eqref{eq:Am-B}, we obtain
\[
\E \|z(\ell_1)\|^2 \le C e^{\frac{C+1}{\ell_2-\ell_1}}
\left(\int_\ell^{\ell_1} \chi_E (\E\|z(t)\|_G^2)^{1/2} dt\right)^2,
\]
which implies the observability inequality \eqref{eq:obs-inq},
completing the proof.
\end{proof}

Next, by the standard duality augment, we have the following
equivalence between the null controllability of the equation
\eqref{eq:main} and the observability inequality for the adjoint
equation \eqref{eq:adjoint}.
\begin{proposition}
  \label{prop:equiv}
  For any $T>0$, the equation \eqref{eq:main} is null controllable at
  time $T$ with the control $u$ in the space of
  $L^\infty_\cF(0,T;L^2(D))$ such that the estimate \eqref{eq:control
    est.} holds if and only if there exists $C>0$ such that the
  solution of the adjoint equation \eqref{eq:adjoint} satisfies the
  observability inequality \eqref{eq:obs-inq}.
\end{proposition}

We omit the proof here, and refer the reader to, for example
\cite[Proposition 1.1]{Liu2014}. Then Theorem \ref{thm:main} is a direct
consequence of Theorem \ref{thm:obs-inq} and Proposition \ref{prop:equiv}.

\section{A Relaxed Optimal Actuator Location Problem}
\label{sec:relaxed}

\subsection{Problem formulation}
\label{subsec:formulation}

In the sequel, we assume $E=[0,T]$.

Now we consider the following norm optimal control problem
\begin{equation}
  \label{eq:norm optimal}
  N(G) = \inf \{ \E \|u\|^2_{L^2((0,T)\times D)} \mid y(T;G,u) = 0\
  \text{in}~ D, \bP\text{-a.s.} \},
\end{equation}
where $y(\cdot;G,u)$ is the solution of equation \eqref{eq:main}. In
the problem \eqref{eq:norm optimal}, we say $u$ is an {\it admissible
  control} if $u \in L^2_\cF(0,T;L^2(D))$ and $y(T;G,u)=0$ in $D$,
$\bP$-a.s.; we say $u^\ast$ is an {\it optimal minimal norm control} if $u^\ast$ is
an admissible control such that $N(G)$ is achieved. 
\begin{remark}
  \label{rem:7.20.1}
It is obvious that minimizing $\E\|u\|^2_{L^2((0,T)\times D)}$ is
equivalent to minimizing $\E\|u\|_{L^2((0,T)\times D)}$. Thus, the
problem we consider is a natural generalization of the usual norm optimal
control problem in the deterministic case.
\end{remark}

Given $\alpha\in (0,1)$, let
\begin{equation}
  \label{eq:W}
  \cW = \{G\subseteq D \mid G~\text{ is Lebesgue measurable with}~
  |G| = \alpha|D|\},
\end{equation}
where  $|\cdot|$ is the Lebesgue measure on
$\R^d$.

A classical optimal actuator location of the minimal norm control
problem is to seek a set $G^\ast\in \cW$ such that
\begin{equation}
  \label{eq:optimal location}
  N(G^\ast) = \inf_{G\in \cW} N(G).
\end{equation}
If such a $G^\ast$ exists, we say that $G^\ast$ is an {\it optimal
actuator location} of the optimal minimal norm controls. Any optimal
minimal norm
control $u^\ast$ satisfying
\[
\E \|u^\ast\chi_{G^\ast}\|^2_{L^2((0,T)\times D)} = N(G^\ast),
\]
is called an {\it optimal control} with respect to the optimal
actuator location $G^\ast$.

The existence of the optimal actuator location $G^\ast$ is generally
not guaranteed because of the absence of the compactness of $\cW$. For
this reason, we consider instead a relaxed problem. To this end, define
\begin{equation}
  \label{eq:B}
  \cB = \left\{\beta \in L^\infty(D;[0,1]) \mid \|\beta\|^2 = \alpha |D|\right\}.
\end{equation}
Note that the set $\cB$ is a relaxation of the set $\{\chi_G\mid G\in
\cW\}$.

For any $\beta\in \cB$, consider the following equation
\begin{equation}
  \label{eq:relaxed eqn}
  \begin{cases}
d y = A y dt + 
\beta u(t) dt + a(t) y dw(t),
& t\in (0,T),\\
y(0) = y_0.
\end{cases}
\end{equation}
We denote by $y(\cdot;\beta,u)$ the solution of equation
\eqref{eq:relaxed eqn}, and say the system \eqref{eq:relaxed eqn} {\it
  null controllable} if there exists $u\in L^2_\cF(0,T;L^2(D))$ such
that $y(T;\beta, u) =0$ in $D$, $\bP$-a.s. Accordingly, the problem
\eqref{eq:norm optimal} is replaced by 
\begin{equation}
  \label{eq:norm optimal relax}
  N(\beta) = \inf \{ \E \|u\|^2_{L^2((0,T)\times D)} \mid y(T;\beta,
  u) =0 ~\text{in}~D, \bP\text{-a.s.}\},
\end{equation}
and the classical optimal actuator location problem \eqref{eq:optimal
  location} is changed into the following relaxed problem
\begin{equation}
  \label{eq:optimal location relax}
  N(\beta^\ast) = \inf_{\beta\in\cB} N(\beta).
\end{equation}
Any solution $\beta^\ast$ to the problem \eqref{eq:optimal location
  relax} is called a {\it relaxed optimal actuator location} of the
optimal minimal norm controls.

Now we study the controllability of the relaxed system
\eqref{eq:relaxed eqn} with the same adjoint equation
\eqref{eq:adjoint}, and make sure that the set on the right hand side
of \eqref{eq:norm optimal relax} is not empty. In fact, the null
controllability is equivalent to the following observability
inequality, as we have done in the proof of Theorem \ref{thm:main}.
\begin{lemma}
  \label{lem:obs relax}
The system \eqref{eq:adjoint} is exactly observable, i.e., there
exists a constant $C>0$, independent of $\beta$, but possibly
depending on $\alpha$ such that for all $\eta\in L^2(\Omega, \cF_T,
\bP;L^2(D))$ and $\beta\in \cB$,
\begin{equation}
  \label{eq:obs relax}
  \E \|z(0;T,\eta)\|^2 \le C\int_0^T \E \|\beta z(t;T,\eta)\|^2 dt.
\end{equation}
\end{lemma}

\begin{proof}
  By Theorem \ref{thm:obs-inq}, for each $G\in \cW$, there exists a
  constant $C>0$ such that the solution of equation \eqref{eq:adjoint}
  satisfies
\begin{equation}
  \label{eq:7.6.1}
  \E \|z(0;T,\eta)\|^2 \le C \int_0^T \|\chi_G z(t;T,\eta)\|^2 dt,
\end{equation}
for all $\eta\in L^2(\Omega,\cF_T,\bP)$. Moreover, the constant $C$
only depends on the measure of the set $G$.

For any $\beta\in \cB$, let
\[
\gamma = \frac{|\{\beta\ge \sqrt{\alpha/2}\}|}{|D|}
\]
Since
\begin{align*}
  \alpha\cdot |D|
& = \int_D \beta^2 dx = \int_{\{\beta\ge \sqrt{\alpha/2}\}} \beta^2 dx +
\int_{\{\beta<\sqrt{\alpha/2}\}} \beta^2 dx\\
& \le |\{\beta\ge \sqrt{\alpha/2}\}| + \frac{\alpha}{2}\cdot |\{\beta< \sqrt{\alpha/2}\}|,
\end{align*}
we have
\[
\gamma \cdot |D| + \frac{\alpha}{2}(1-\gamma)\cdot |D|,
\]
and consequently,
\[
\gamma \ge \frac{\alpha}{2-\alpha}.
\]
Therefore, we obtain 
\begin{equation}
  \label{eq:7.6.2}
  |\{\beta\ge \sqrt{\alpha/2}\}| \ge \frac{\alpha}{2-\alpha} \cdot |D|,
\end{equation}
for all $\beta\in \cB$.
It then follows from inequality \eqref{eq:7.6.1} with $G= \{\beta\ge
\sqrt{\alpha/2}\}$ that
\begin{align*}
  \E \|z(0;T,\eta)\|^2 
& \le C \E \int_0^T \int_D \chi_{\{\beta\ge
\sqrt{\alpha/2}\}} z^2(t;T,\eta) dx dt\\
& \le C \E \int_0^T \int_D \chi_{\{\beta\ge
\sqrt{\alpha/2}\}}\left(\frac{\beta}{\sqrt{\alpha/2}}\right)^2
z^2(t;T,\eta) dx dt\\
& = C \cdot \frac{2}{\alpha}\cdot \E \int_0^T \int_D \chi_{\{\beta\ge
\sqrt{\alpha/2}\}} \beta^2 z^2(t;T,\eta) dx dt\\
& \le C \int_0^T \E \|\beta z(t;T,\eta)\|^2 dt,
\end{align*}
which completes the proof.
\end{proof}
\begin{remark}
The observability inequality \eqref{eq:obs relax} is in fact an $L^2$ estimate, and it is sufficient for our purpose in this section, though we have an $L^1$ estimate in \eqref{eq:obs-inq}.
\end{remark}
\subsection{The optimal minimal norm control}
\label{ss:optimal control}
In general, it is not easy (or impossible) to solve the problem \eqref{eq:norm optimal
  relax} directly; see \cite{Guo2014} for a special class of
subdomains. Instead, let us introduce a functional
\begin{equation}
  \label{eq:J-z-beta}
  \cJ(\eta;\beta) = \frac{1}{2}\int_0^T \E \|\beta z(t;\eta)\|^2 dt +
  \E(y_0, z(0;\eta)),
\end{equation}
and propose the following variational problem
\begin{equation}
  \label{eq:J-beta}
  \cJ(\beta) = \inf_{\eta\in L^2(\Omega,\cF_T,\bP;L^2(D))} \cJ(\eta;\beta).
\end{equation}
Here and what follows, we simply set $z(\cdot;\eta) = z(\cdot;T,\eta)$
for the solution of the adjoint equation \eqref{eq:adjoint} with the
terminal condition $z(T) = \eta$. We will show later the equivalence between the problem
\eqref{eq:J-beta} and the problem \eqref{eq:norm optimal relax}.

To this end, denote by
\begin{equation}
  \label{eq:X}
  X = \{z(\cdot;\eta)\mid \eta\in L^2(\Omega,\cF_T,\bP;L^2(D))\},
\end{equation}
and for each $\beta\in \cB$, define $F_\beta: X\to \R$ by
\begin{equation}
  \label{eq:F-beta}
  F_\beta(z) = \left(\E \int_0^T \|\beta z\|^2 dt\right)^{1/2}.
\end{equation}
It follows from the observability inequality \eqref{eq:obs relax} that
$F_\beta$ is indeed a norm on space $X$. We denote by $\ol{X_\beta}$
the completion of the space $X$ under the norm $F_\beta$. The
following proposition provides us a description of $\ol{X_\beta}$.

\begin{lemma}
  \label{lem:X-beta}
  Under an isomorphism, any element of $\ol{X_\beta}$ can be expressed
  as a process $\varphi \in L^2_\cF(\Omega;C([0,T);L^2(D)))$, which
  satisfies
\begin{equation}
  \label{eq:7.6.3}
  d\varphi = -A\varphi dt - a(t)Z dt + Zdw(t)
\end{equation}
 for some $Z\in L^2_\cF(0,T;L^2(D))$ in $L^2(0,T;L^2(D))$, $\bP$-a.s. Moreover, $\beta\varphi =
\lim_{n\to\infty}\beta z(\cdot;\eta_n)$ for some sequence
$\{\eta_n\}\subseteq L^2(\Omega,\cF_T,\bP)$ in
$L^2(\Omega;L^2((0,T)\times D))$.
\end{lemma}

\begin{proof}
  Let $\ol{\varphi}\in (\ol{X_\beta},\ol{F_\beta})$, where
    $(\ol{X_\beta},\ol{F_\beta})$ is the completion of $(X,
    F_\beta)$. Then there exists a sequence $\{\eta_n\}\subseteq
    L^2(\Omega,\cF_T,\bP)$ such that
\[
\ol{F_\beta}(z(\cdot;\eta_n)-\ol{\varphi})\to 0,\ \text{as}~n\to\infty,
\]
from which, one has
\[
F_\beta(z(\cdot;\eta_n)-z(\cdot;\eta_m)) =
\ol{F_\beta}(z(\cdot;\eta_n)-z(\cdot;\eta_m))\to 0\ \text{as}~n, m
\to \infty.
\]
In other words,
\begin{equation}
  \label{eq:7.6.4}
  \E \int_0^T \|\beta z(t;\eta_n)-\beta z(t;\eta_m)\|^2 dt \to
  0\ \text{as}~n, m \to \infty.
\end{equation}
Hence, there exists $\hat{\varphi}\in L^2_\cF(0,T;L^2(D))$ such that
\begin{equation}
  \label{eq:7.6.5}
  \beta z(\cdot,\eta_n)\to \hat{\varphi}~\text{strongly
    in}~L^2(\Omega;(0,T)\times D).
\end{equation}
Now choose a strictly increasing sequence $\{T_k\}\subseteq (0,T)$
such that $T_k\to T$ as $k\to\infty$. Set $(z_n, Z_n) =
(z(\cdot;\eta_n), Z_n(\cdot;\eta_n))$, i.e., the solution of equation
\eqref{eq:adjoint} with the terminal condition $z_n(T)=\eta_n$.

(a) For $T_1$. By the observability inequality \eqref{eq:obs relax}
and \eqref{eq:7.6.4},
\[
  \E \|z(T_2,\eta_n)\|^2 
\le C_1 \E \int_{T_2}^T \|\beta z(t;\eta_n)\|^2 dt\le C_1 \E
\int_0^T \|\beta z(t;\eta_n)\|^2 dt\le C_1,
\]
for all $n\ge 1$. Then there exist a subsequence
$\{z(T_2,\eta_{1n})\}$ of $\{z(T_2,\eta_n)\}$ and $z_{T_2,1}\in
L^2(\Omega,\cF_{T_2},\bP)$ such that
\[
z(T_2,\eta_{1n})\to z_{T_2,1}\ \text{weakly in}~ L^2(\Omega\times D).
\]
Consequently, there exist a subsequence $\{(z_{1n},Z_{1n})\}$ of $\{(z_n, Z_n)\}$
and 
$(\varphi_1, \psi_1)$ in the space of  $L^2_\cF(\Omega;C([0,T_2];L^2(D))) \times
L^2_\cF(0,T_2;L^2(D))$ solving the adjoint equation \eqref{eq:adjoint} with the terminal
conditions $z_{1n}(T_2) = z(T_2,\eta_{1n})$ and $\varphi_1(T_2) =
z_{T_2,1}$, respectively, and
\[
(z_{1n},Z_{1n})\to (\varphi_1,\psi_1)\ \text{weakly
  in}~L^2(\Omega;C([0,T_2];L^2(D)))\times L^2(\Omega;(0,T_2)\times D).
\]
In particular,
\begin{equation}
  \label{eq:7.6.6}
  (z_{1n},Z_{1n})\to (\varphi_1,\psi_1)\ \text{weakly
    in}~L^2(\Omega;C([0,T_1];L^2(D)))\times L^2(\Omega;(0,T_1)\times D),
\end{equation}
and
\begin{equation}
  \label{eq:7.6.7}
  \beta z_{1n}\to \beta \varphi_1\ \text{weakly
    in}~L^2(\Omega;(0,T_1)\times D).
\end{equation}
Thus, it follows from \eqref{eq:7.6.5} and \eqref{eq:7.6.7} that
\[
\beta \varphi_1 = \hat{\varphi}\ \text{in}~L^2(\Omega;(0,T_1)\times D).
\]

(b) For $T_2$. In the same spirit of (a), we can find a subsequence
$\{(z_{2n},Z_{2n})\}$ of $\{(z_{1n},Z_{1n})\}$, and 
$(\varphi_2, \psi_2)$ in the space of  $L^2_\cF(\Omega;C([0,T_3];L^2(D))) \times
L^2_\cF(0,T_3;L^2(D))$ solving the adjoint equation \eqref{eq:adjoint} with the terminal
conditions $z_{2n}(T_3) = z(T_3,\eta_{2n})$ and $\varphi_1(T_3) =
z_{T_3,2}$, respectively, and
\[
(z_{2n},Z_{2n})\to (\varphi_2,\psi_2)\ \text{weakly
  in}~L^2(\Omega;C([0,T_3];L^2(D)))\times L^2(\Omega;(0,T_3)\times D),
\]
where $\{\eta_{2n}\}$ is a subsequence of $\{\eta_{1n}\}$ such that
$z(T_3;\eta_{2n})$ converges weakly to $z_{T_3,2}$ in
$L^2(\Omega\times D)$. Then it follows from \eqref{eq:7.6.5},
\eqref{eq:7.6.6} and \eqref{eq:7.6.7} that 
\[
(\varphi_2,\psi_2)\restriction_{[0,T_1]} = (\varphi_1,\psi_1),
\]
and
\[
\beta \varphi_2 = \hat{\varphi}~\text{in}~L^2(\Omega;(0,T_2)\times D).
\]

(c) In general, we obtain a sequence $\{(\varphi_k,\psi_k)\}$
satisfies for each $k\ge 1$ that
\begin{itemize}
  \item $\{(\varphi_k,\psi_k)\}\in L^2_\cF(\Omega;C([0,T_{k+1}];L^2(D))) \times
L^2_\cF(0,T_{k+1};L^2(D))$;
\item $(\varphi_{k+1},\psi_{k+1})\restriction_{[0,T_k]} =
  (\varphi_k,\psi_k)$;
\item $\{(\varphi_k,\psi_k)\}$ satisfies \eqref{eq:7.6.3} on
  $(0,T_{k+1})$;
\item $\beta \varphi_k = \hat{\varphi}$ in $L^2(\Omega;(0,T_k)\times D)$.
\end{itemize}

Now define
\[
(\varphi(t), Z(t)) = (\varphi_k(t), \psi_k(t)), \ t\in [0,T_k].
\]
Then $(\varphi(t), Z(t))\in L^2_\cF(\Omega;C([0,T);L^2(D))) \times
L^2_\cF(0,T;L^2(D))$ satisfies equation \eqref{eq:7.6.3}, and
\[
\beta \varphi = \hat{\varphi} = \lim_{n\to \infty} \beta z(\cdot;\eta_n).
\]
Under an isometric isomorphism, we can identify $\ol{\varphi}$ by
$\varphi$. The proof is completed.
\end{proof}
\begin{remark}
  \label{re:beta-phi}
  The element $\varphi$ in $\ol{X_\beta}$ is not necessarily in the space
  of $L^2_\cF(0,T;L^2(D))$, but $\beta \varphi\in
L^2_\cF(0,T;L^2(D))$ for $\beta\in \cB$. Also, because of the
isomorphism, we can write $\ol{F_\beta}(\varphi) = \left(\E\int_0^T
  \|\beta \varphi\|^2dt\right)^{1/2}$.
\end{remark}

Next, let us introduce an auxiliary operator
\begin{equation}
  \label{eq:T-beta}
  \cT_\beta: \beta \ol{X_\beta}\subseteq L^2_\cF(0,T;L^2(D))\to
  L^2(\Omega;L^2(D)),\quad \beta \varphi\to \varphi(0). 
\end{equation}
By Lemma \ref{lem:X-beta}, the operator $\cT_\beta$ is well
defined, and it is bounded as well. In fact, if we consider the
equation \eqref{eq:adjoint} on the interval $(0,T/2)$, then by the
observability inequality \eqref{eq:obs relax}, we have
\[
\E \|\varphi(0)\|^2 \le C\E \int_0^{T/2} \|\beta \varphi\|^2 dt\le C \E \int_0^T \|\beta \varphi\|^2 dt.
\]
Then the functional $\cJ(\beta)$ defined in \eqref{eq:J-beta} can be
written as
\begin{align*}
  \cJ(\beta)
& = \inf_{z\in X}\left[\frac{1}{2}\int_0^T \E \|\beta z\|^2dt +
  \E(y_0,z(0))\right]\\
& = \inf_{\varphi\in\ol{X_\beta}}\left[\frac{1}{2}\int_0^T \E \|\beta \varphi\|^2dt +
  \E(y_0,\cT_\beta(\beta \varphi))\right]\\
& = \inf_{\varphi\in\ol{X_\beta}}\left[\frac{1}{2}\int_0^T \E \|\beta \varphi\|^2dt +
  \int_0^T \E(\cT^\ast_\beta y_0,\beta \varphi) dt\right].
\end{align*}
Set $y_{0,\beta} = \cT^\ast_\beta y_0$, and thus the problem
\eqref{eq:J-beta} is equivalent to the following problem
\begin{equation}
  \label{eq:J-beta-new}
  \mathcal{V}(\beta) = \inf_{\varphi\in\ol{X_\beta}}\left[\frac{1}{2}\int_0^T \E \|\beta \varphi\|^2dt +
  \int_0^T \E(y_{0,\beta},\beta \varphi) dt\right].
\end{equation}
The key motivation of this transformation is that the functional on the right hand side of the problem
  \eqref{eq:J-beta-new} is coercive in $\varphi$ with
respect to the norm $\ol{F_\beta}$, but in general, $\cJ(\eta;\beta)$ in
\eqref{eq:J-z-beta} does not satisfy such a condition. The next
theorem characterizes the minmal norm control of problem
\eqref{eq:norm optimal relax} in terms of the
solution of the problem \eqref{eq:J-beta-new}. 
\begin{theorem}
  \label{thm:u star}
Fix $\beta\in \cB$. Suppose $y_0\in L^2(\Omega,\cF_0,\bP;L^2(D))$. Then problem \eqref{eq:J-beta-new} admits a unique
solution $\varphi^\ast$. Moreover, the control defined by
\begin{equation}
  \label{eq:u star}
  u^\ast = \beta \varphi^\ast
\end{equation}
is the minimal norm optimal control to the problem \eqref{eq:norm
  optimal relax}, and 
\begin{equation}
  \label{eq:u star norm}
  N(\beta) = \int_0^T \E \|\beta \varphi^\ast\|^2 dt.
\end{equation}
\end{theorem}

\begin{proof}
  It is obvious that the functional on the right hand side of
  \eqref{eq:J-beta-new} is continuous, strictly convex and coercive
in $\varphi$ with respect to the norm $\ol{F_\beta}$. Therefore, the
problem \eqref{eq:J-beta-new} admits a unique solution, denoted by
$\varphi^\ast$.  

It follows from Lemma \ref{lem:X-beta} that the control $u^\ast
= \beta \varphi^\ast$ is well defined and $u^\ast \in
L^2_\cF(0,T;L^2(D))$. We claim first that $u^\ast$ is a control
driving the solution $y$ of equation \eqref{eq:relaxed eqn} to rest at
time $T$. In fact, by the optimality of $\varphi^\ast$, we obtain the
following Euler-Lagrange equation to the variational problem
\eqref{eq:J-beta-new}:
\begin{equation}
  \label{eq:EL}
  \int_0^T \E (u^\ast, \beta\psi) dt + \int_0^T \E (y_{0,\beta},\beta\psi) dt =
  0, \ \text{for all}~ \psi\in \ol{X_\beta}.
\end{equation}
Taking $\psi = z(\cdot;\eta)\in X$ for any $\eta\in
L^2(\Omega;L^2(D))$, a straightforward computation and \Ito formula imply that
\[
y(T;\beta,u^\ast) = 0\ \text{in}~D, \bP\text{-a.s.}
\]

Next, we will show that $u^\ast$ is optimal in the sense that
\begin{equation}
  \label{eq:7.8.1}
\int_0^T\E \|u^\ast\|^2 dt \le \int_0^T \E \|\hat{u}\|^2 dt,
\end{equation}
for any $\hat{u}\in L^2_\cF(0,T;L^2(D))$ such that $y(T;\beta,\hat{u})
= 0$ in $D$, $\bP$-a.s. 

Without loss of generality, we assume
$u^\ast\neq 0$. By \Ito formula, we have
\[
\int_0^T \E (\hat{u}, \beta z) dt + \E (y_0, z(0)) =
  0, \ \text{for all}~ z\in X,
\]
or equivalently
\[
\int_0^T \E (\hat{u}, \beta z) dt + \int_0^T \E (y_{0,\beta},\beta z) dt =
  0, \ \text{for all}~ z\in X,
\]
which, together with equality \eqref{eq:EL}, implies
\begin{equation}
  \label{eq:7.8.2}
\int_0^T \E (u^\ast, \beta z) dt = \int_0^T \E (\hat{u}, \beta z) dt, \ \text{for all}~ z\in X.
\end{equation}
By the density argument, the equality \eqref{eq:7.8.2} still holds for
all $\psi\in \ol{X_\beta}$. Thus, replacing $z$ in \eqref{eq:7.8.2} by
$\varphi^\ast$ gives 
\[
\int_0^T \E \|u^\ast\|^2 dt = \int_0^T \E (u^\ast,\hat{u})dt\le
\left(\int_0^T\E \|u^\ast\|^2dt\right)^{1/2}\left(\int_0^T\E \|\hat{u}\|^2dt\right)^{1/2}.
\]
Therefore, the inequality \eqref{eq:7.8.1} is true and this concludes
the proof.
\end{proof}

From above, we can describe the relation between $\mathcal{V}(\beta)$ (or
equivalently $\cJ(\beta)$) and $N(\beta)$.

\begin{corollary}
  Let $\beta\in \cB$ and $y_0\in L^2(\Omega,\cF_0,\bP;L^2(D))$. Then
\begin{equation}
  \label{eq:V-N}
  \mathcal{V}(\beta) = -\frac{1}{2}N(\beta),
\end{equation}
where $\mathcal{V}(\beta)$ and $N(\beta)$ are defined as in
\eqref{eq:J-beta-new}, and \eqref{eq:norm optimal relax}, respectively.
\end{corollary}
\begin{proof}
  Let $\varphi^\ast$ be a solution of the problem
  \eqref{eq:J-beta-new} such that
\[
\mathcal{V}(\beta) = \frac{1}{2}\int_0^T \E \|\beta \varphi^\ast\|^2dt +
  \int_0^T \E(y_{0,\beta},\beta \varphi^\ast) dt.
\]
On the other hand, it follows from the Euler-Lagrange equation
\eqref{eq:EL} that
\[
\int_0^T \E(y_{0,\beta},\beta \varphi^\ast) dt = - \int_0^T \E \|\beta \varphi^\ast\|^2dt.
\]
Thus, by \eqref{eq:u star norm} we have
\[
\mathcal{V}(\beta) = -\frac{1}{2}\int_0^T \E \|\beta
\varphi^\ast\|^2dt = -\frac{1}{2}N(\beta).
\]
\end{proof}

\subsection{Existence of relaxed optimal actuator location}
\label{ss:relaxed optimal}
Now we are ready to show the existence of relaxed optimal actuator
location of the optimal minimal norm controls, i.e., we can find
$\beta^\ast\in\cB$ such that $N(\beta^\ast) = \inf_{\beta\in\cB}
N(\beta)$. To this end, define
\begin{equation}
  \label{eq:theta}
  \Theta = \left\{\theta\in L^\infty(D;[0,1]) \mid \int_D \theta(x) dx
      = \alpha |D|\right\}.
\end{equation}
It is clear that
\begin{equation}
  \label{eq:theta-beta}
  \beta^2\in \Theta~\text{for any}~\beta\in\cB,
  ~\text{and}~\theta^{1/2}\in\cB~\text{for all}~\theta\in\Theta.
\end{equation}
Then it follows from the relation \eqref{eq:V-N} that
\begin{align*}
  & \inf_{\beta\in\cB} \frac{1}{2}N(\beta) =
  \inf_{\beta\in\cB}-\mathcal{V}(\beta)
  =\inf_{\beta\in\cB}-\cJ(\beta)\\
= & \inf_{\beta\in\cB} \sup_{z\in X} \left[\frac{1}{2}\int_0^T \E
  \|\beta z\|^2 dt + \E (y_0,z(0))\right]\\
= & \inf_{\theta\in\Theta}\sup_{z\in X}\left[- \frac{1}{2}\E\int_0^T
  \int_D\theta z^2 dxdt - \E(y_0,z(0))\right]\\
=: & \inf_{\theta\in\Theta} \sup_{z\in X} F(\theta,z),
\end{align*} 
where the functional $F$ is defined by
\begin{equation}
  \label{eq:F}
  F(\theta,z) = - \frac{1}{2}\E\int_0^T\int_D
  \theta z^2 dxdt - \E(y_0,z(0)).
\end{equation}
Therefore, seeking a minimizer $\beta^\ast\in\cB$ for $N(\beta)$ amounts
to finding a minimizer $\theta^\ast\in\Theta$ for $\sup_{z\in
  X}F(\theta,z)$.

Let us equip $L^\infty(D)$ with the weak$^\ast$ topology. Then
$\Theta$ is compact in $L^\infty(D)$. 
\begin{lemma}
  \label{lem:lsc}
Given $y_0\in L^2(\Omega;\cF_0,\bP;L^2(D))$ and $z\in X$. Then the functional
$F(\cdot,z):\Theta\to\R\cup\{+\infty\}$ defined in \eqref{eq:F} is
sequentially weakly$^\ast$ lower semi-continuous.
\end{lemma}
\begin{proof}
  Suppose there is a sequence $\{\theta_n\}\subseteq \Theta$ such that
\[
\theta_n\to \theta ~\text{weakly}^\ast~\text{in}~L^\infty(D).
\]
Then for any $t\in [0,T]$, we have
\[
\lim_{n\to\infty} \E \int_D \theta_n z^2(t) dx = \E \int_D \theta z^2(t)dx \le \E \|z(t)\|^2.
\]
Since $\int_0^T \E \|z(t)\|^2 dt<\infty$, it follows from the
Dominated Convergence Theorem, and \eqref{eq:F} that
\[
\lim_{n\to\infty} F(\theta_n,z) = F(\theta,z).
\]
So $F(\cdot,z)$ is sequentially weakly$^\ast$ continuous, and in particular, lower semi-continuous.
\end{proof}

It is obvious that the
functional $F(\cdot,z)$ is linear in $\theta$ for any $z\in X$, so it
is convex. Then it follows from Proposition 2.31 in
\cite[page 62]{Ambrosio2000} that $F(\cdot,z)$ is weakly$^\ast$ lower
semi-continuous. Under the weak$^\ast$ topology in $L^\infty(D)$,
$F(\cdot, z)$ is lower semi-continuous, so is $\sup_{z\in X}
F(\cdot,z)$. Together with the fact that $\Theta$ is compact in
$L^\infty(D)$, we claim that there exists $\theta^\ast\in\Theta$
minimizing $\sup_{z\in X} F(\cdot,z)$ by Theorem 38.B in
\cite [page 152]{Zeidler1985}. Equivalently, we obtain the following
theorem of existence to conclude this subsection.
\begin{theorem}
  \label{thm:beta-star}
  Suppose $y_0\in L^2(\Omega,\cF_0,\bP;L^2(D))$. Then the problem
  \eqref{eq:optimal location relax} admits a solution $\beta^\ast\in
  \cB$, i.e., 
\[
N(\beta^\ast) = \inf_{\beta\in \cB} N(\beta).
\]
\end{theorem}

\subsection{Characterization via Nash equilibrium}
\label{sec:char-via-nash}
Now we define a non-negative nonlinear functional $F_\Theta$ on $X$ by
\begin{equation}
  \label{eq:F-theta}
  F_\Theta(z):= \sup_{\theta\in\Theta} F_{\theta^{1/2}}(z),\ z\in X
\end{equation}
where $F_{\theta^{1/2}}$ is defined as in \eqref{eq:F-beta}.
Since $F_{\theta^{1/2}}$ is a norm on $X$ for each $\theta\in \Theta$,
$F_\Theta$ is also a norm on $X$. Thus, $(X, F_\Theta)$ is a normed
space, and we denote by $(\ol{X_\Theta},\ol{F_\Theta})$ its
completion.

Along the same line in the proof of Lemma \ref{lem:X-beta}, we
have the following similar result.
\begin{lemma}
  \label{lem:X-theta}
  Under an isomorphism, any element of $\ol{X_\Theta}$ can be expressed
  as a process $\varphi \in L^2_\cF(\Omega;C([0,T);L^2(D)))$, which
  satisfies
\begin{equation}
  \label{eq:7.12.5}
  d\varphi = -A\varphi dt - a(t)Z dt + Zdw(t)
\end{equation}
 for some $Z\in L^2_\cF(0,T;L^2(D))$ in $L^2(0,T;L^2(D))$, $\bP$-a.s. Moreover, $F_\Theta(\varphi) =
\lim_{n\to\infty}F_\Theta(z(\cdot;\eta_n))$ for some sequence
$\{\eta_n\}\subseteq L^2(\Omega,\cF_T,\bP)$.
\end{lemma}

By Lemma \ref{lem:X-theta}, we have the following inclusion relation:
\begin{equation}
  \label{eq:7.11.4}
  \ol{X_\Theta}\subseteq L^2_\cF(0,T;L^2(D)).
\end{equation}
In fact, suppose that $n_0\in \N$ so that $n_0\ge 1/\alpha$. Then there
are $n_0$ measurable subsets $G_1,\cdots, G_{n_0}$ of $D$ such that
\[
G_j\in \cW,\ 1\le j\le n_0,\ \text{and}~\bigcup_{j=1}^{n_0} G_j = D.
\]
Then the inclusion relation follows from
\begin{align*}
  \int_0^T \E \|\varphi\|^2 dt
& = \int_0^T \E \left\|\varphi\sum_{j=1}^{n_0}\chi_{G_j}\right\|^2
dt\\
& \le n_0 \sum_{j=1}^{n_0}\int_0^T \E \|\varphi\chi_{G_j}\|^2 dt\\
& \le n_0 \sum_{j=1}^{n_0} \ol{F_\Theta}^2(\varphi) = n_0^2 \ol{F_\Theta}^2(\varphi).
\end{align*}
On the other hand, it is obvious that $\ol{F_\Theta}(\varphi)\le \int_0^T
\E \|\varphi\|^2 dt$. Thus, $F_\Theta$ and
$\left(\E\|\cdot\|^2_{L^2((0,T)\times D)}\right)^{1/2}$ are equivalent
norms on $X$.

In this subsection, we solve the following Nash equilibrium problem of
two-person zero-sum game (see Appendix): to find $\bar{\theta}\in
\Theta, \bar{\varphi}\in \ol{X_\Theta}$ such that
\begin{equation}
  \label{eq:Nash}
  F(\bar{\theta},\bar{\varphi}) = \sup_{\varphi\in
   \ol{X_\Theta}}F(\bar{\theta},\varphi)=\inf_{\theta\in\Theta}F(\theta,\bar{\varphi}),
\end{equation}
where $F(\theta,\varphi)$ is defined as in \eqref{eq:F}. This requires
by Theorem \ref{thm:Nash} in Appendix that we solve the following two problems
\begin{equation}
  \label{eq:7.12.1}
  \inf_{\theta\in\Theta}\sup_{\varphi\in\ol{X_\Theta}}F(\theta,\varphi),
\end{equation}
and
\begin{equation}
  \label{eq:7.12.2}
  \sup_{\varphi\in\ol{X_\Theta}}\inf_{\theta\in\Theta}F(\theta,\varphi),
\end{equation}
and verify the equality \eqref{eq:Nash}.

In fact, the problem \eqref{eq:7.12.1} is solved by choosing
$\bar{\theta}= (\beta^*)^2$, where $\beta^\ast\in\cB$ is a solution of the
problem \eqref{eq:optimal location relax}, guaranteed by Theorem
\ref{thm:beta-star}. To see this clearly, recall that $X$ is dense in
$\ol{X_\beta}$, and for each $\theta\in\Theta$ with $\theta=\beta^2$
\[
\sup_{\varphi\in X} F(\theta, \varphi) = \sup_{\varphi\in\ol{X_\beta}}F(\theta,\varphi),
\]
where we use the fact that $F(\theta,\cdot)$ is continuous with
respect to the norm $\ol{F_\beta}$, the completion of $F_\beta$ in
\eqref{eq:F-beta}. On the other hand, since for each $\beta\in\cB$, we have
\[
\int_0^T \E\|\beta z\|^2 dt \le F^2_\Theta(z),\ \forall z\in X,
\]
which implies
\[
\ol{X_\Theta}\subseteq \ol{X_\beta},\ \forall \beta\in\cB.
\]
Therefore,
\begin{equation}
  \label{eq:7.12.7}
\sup_{z\in
  X}F(\theta,z)\le\sup_{\varphi\in\ol{X_\Theta}}F(\theta,\varphi)
\le\sup_{\varphi\in\ol{X_\beta}}F(\beta^2,\varphi) =\sup_{z\in X}F(\theta,z),
\end{equation}
and thus
\[
\inf_{\beta\in\cB}\frac{1}{2}N(\beta) = \inf_{\beta\in\cB}-\cJ(\beta)
= \inf_{\theta\in\Theta}\sup_{z\in X}F(\theta,z) = \inf_{\theta\in\Theta}\sup_{\varphi\in\ol{X_\Theta}}F(\theta,\varphi).
\]
So the problem \eqref{eq:7.12.1} is solved by Theorem
\ref{thm:beta-star}.

To solve the problem \eqref{eq:7.12.2} is to find
$\bar{\varphi}\in\ol{X_\Theta}$ such that
\[
\inf_{\theta\in\Theta}F(\theta,\bar{\varphi}) = \sup_{\varphi\in\ol{X_\Theta}}\inf_{\theta\in\Theta}F(\theta,\varphi),
\]
or equivalently,
\begin{equation}
  \label{eq:7.12.3}
  \sup_{\theta\in\Theta}-F(\theta,\bar{\varphi}) = \inf_{\varphi\in\ol{X_\Theta}}\sup_{\theta\in\Theta}-F(\theta,\varphi).
\end{equation}
\begin{lemma}
  \label{lem:-F}
For any $y_0\in L^2(\Omega,\cF_0,\bP;L^2(D))$, the problem \eqref{eq:7.12.3}
admits a unique solution.
\end{lemma}
\begin{proof}
Define the functional $F:\ol{X_\Theta}\to \R$ by
\[
F(\varphi):=\sup_{\theta\in\Theta}-F(\theta,\varphi).
\]
Then
\[
F(\varphi) =\sup_{\theta\in\Theta}\frac{1}{2}\int_0^T
\E \|\theta^{1/2}\varphi\|^2 dt + \E(y_0,\varphi(0))= \ol{F_\Theta}^2(\varphi)+\E(y_0,\varphi(0)).
\]
It is clear that $F$ is strictly convex in $\varphi$. To show
continuity and coercivity, we consider the equation \eqref{eq:7.12.5}
on the time interval $(0,T/2)$. Then by the observability inequality
\eqref{eq:obs relax}, we have for all $\beta\in\cB$ and $\varphi\in\ol{X_\Theta}$
\begin{equation}
  \label{eq:7.12.6}
\E \|\varphi(0)\|^2 \le C\E \int_0^{T/2} \|\beta \varphi\|^2 dt\le C
\E \int_0^T \|\beta \varphi\|^2 dt\le C \ol{F_\Theta}^2(\varphi).
\end{equation}
Thus, by Cauchy-Schwartz inequality, we have $|\E(y_0,\varphi(0))|\le C\ol{F_\Theta}(\varphi)$. Now suppose
there exists a sequence $\{\varphi_n\}\subseteq \ol{X_\Theta}$ such
that $\varphi_n\to \varphi$ in $\ol{X_\Theta}$, i.e.,
$\ol{F_\Theta}(\varphi_n-\varphi)\to 0$, then
\begin{align*}
  |F(\varphi_n)-F(\varphi)|
& \le \left|\ol{F_\Theta}^2(\varphi_n)-\ol{F_\Theta}^2(\varphi)\right|
+ \left|\E(y_0,(\varphi_n-\varphi)(0))\right|\\
& \le C\left|\ol{F_\Theta}(\varphi_n)-\ol{F_\Theta}(\varphi)\right| +
C\ol{F_\Theta}(\varphi_n-\varphi)\\
& \le C\ol{F_\Theta}(\varphi_n-\varphi)\to 0,
\end{align*}
which implies that $F$ is continuous. Finally, it follows from
\eqref{eq:7.12.6} that
\[
F(\varphi) \ge \ol{F_\Theta}^2(\varphi) - C\ol{F_\Theta}(\varphi),
\]
and so $F$ is coercive. Hence, the problem \eqref{eq:7.12.3} has a
unique solution.
\end{proof}

Now it remains to show the equality \eqref{eq:Nash} holds. To this
end, denote by
\begin{equation}
  \label{eq:U+-}
  U^+ = \inf_{\theta\in\Theta} \sup_{z\in X} F(\theta,z),\ U^- =
  \sup_{z\in X} \inf_{\theta\in \Theta} F(\theta, z), 
\end{equation}
where $F$ is defined in \eqref{eq:F}. Let $\cK$ be the collection
of all the finite subsets of $X$. For any $K\in\cK$, set
\begin{equation}
  \label{eq:U-K}
  U_K = \inf_{\theta\in\Theta}\sup_{z\in K} F(\theta, z),\
  \hat{U}:=\sup_{K\in\cK} U_K.
\end{equation}
Then it is easy to verify that
\begin{equation}
  \label{eq:7.11.1}
  U^- \le \hat{U} \le U^+.
\end{equation}
Furthermore, we can obtain the equalities in \eqref{eq:7.11.1}.
\begin{proposition}
  Define $U^-$ and $U^+$ as in \eqref{eq:U+-}, then
\[
U^- = U^+.
\]
\end{proposition}
\begin{proof}
We first show that $U^+\le \hat{U}$.

  Given any $K\in\cK$, using a similar argument to the one above Theorem
  \ref{thm:beta-star}, we can find $\theta_K\in\Theta$ such that
\[
\sup_{z\in K}F(\theta_K,z) = \inf_{\theta\in\Theta}\sup_{z\in
  K}F(\theta,z) = U_K.
\]

This, together with the definition of $\hat{U}$ in \eqref{eq:U-K},
enables us to derive
\begin{equation}
  \label{eq:7.11.2}
  F(\theta_K,z) \le U_K \le \hat{U},\ \text{for all}~z\in K.
\end{equation}
Let $z\in X$, define
\[
S_z:= \{\theta\in\Theta \mid F(\theta,z)\le \hat{U}\}.
\]
It follows from \eqref{eq:7.11.2} that the set $S_z$ is not empty, and 
\begin{equation}
  \label{eq:7.11.3}
  \{\theta_K\} \subseteq \bigcap_{z\in K} S_z \neq \emptyset.
\end{equation}
In addition, since $F(\cdot, z)$ is weakly$^\ast$ lower
semi-continuous, $S_z$ is weakly$^\ast$ closed in $L^\infty(D)$. By
the compactness of $\Theta$ under the weak$^\ast$ topology of
$L^\infty(D)$, we have
\[
\bigcap_{z\in X} S_z \neq \emptyset.
\]
Thus, there exists $\hat{\theta}\in\Theta$ such that $
\sup_{z\in Z}F(\hat{\theta},z)\le\hat{U}$,
and so
\[
U^+ = \inf_{\theta\in\Theta}\sup_{z\in X} F(\theta,z) \le \hat{U}.
\]

Next, we show $U^- = \hat{U}$.

It is clear that both $\Theta$ and $X$ are convex sets. Note that $F(\theta,\cdot)$ is convex for each
$\theta\in\Theta$ and $F(\cdot, z)$ is convex (in fact, it is linear)
for each $z\in X$. By Proposition 8.3 in \cite[page 132]{Aubin1993},
$U^- = \hat{U}$.

Therefore, we have
\[
U^+\le U^-,
\]
which, together with \eqref{eq:7.11.1} implies $U^- = U^+$.
\end{proof}

Again, it follows from \eqref{eq:7.12.7} that 
\[
U^+ = \inf_{\theta\in\Theta} \sup_{\varphi\in\ol{X_\Theta}}F(\theta,\varphi).
\]
On the other hand,
\[
U^-=\sup_{z\in X}\inf_{\theta\in\Theta}F(\theta,z) =
\sup_{\varphi\in\ol{X_\Theta}}\inf_{\theta\in\Theta}F(\theta,\varphi)
\le
\inf_{\theta\in\Theta}\sup_{\varphi\in\ol{X_\Theta}}F(\theta,\varphi),
\]
and thus by $U^-=U^+$ we obtain
\[
\sup_{\varphi\in\ol{X_\Theta}}\inf_{\theta\in\Theta}F(\theta,\varphi)
=
\inf_{\theta\in\Theta}\sup_{\varphi\in\ol{X_\Theta}}F(\theta,\varphi).
\]

Summarizing the previous analysis, we arrive at the following theorem.
\begin{theorem}
  The problem \eqref{eq:Nash} admits at least one Nash
  equilibrium. Specifically, $\bar{\varphi}$ is a solution of the
  problem \eqref{eq:7.12.3} and $\beta^\ast$ is a solution of the
  relaxed optimal actuator location problem \eqref{eq:optimal location
    relax}, if and only if the pair
  $(\bar{\theta}=(\beta^\ast)^2,\bar{\varphi})$ is a Nash equilibrium.
\end{theorem}

Consequently, we can characterize the solution of the relaxed optimal
location problem \eqref{eq:optimal location relax} via a Nash equilibrium.
\begin{theorem}
  There exists at least one solution of the problem \eqref{eq:optimal location
    relax}. In addition, $\beta^\ast$ is a relaxed optimal actuator
  location of the optimal minimal norm controls if and only if there
  exists $\bar{\varphi}\in\ol{X_\Theta}$ such that the pair
  $(\beta^\ast,\bar{\varphi})$ is a Nash equilibrium
  of the following two-person zero-sum game problem: to find
  $(\beta^\ast,\bar{\varphi})\in\cB\times\ol{X_\Theta}$ such that
\begin{align}
  \label{eq:Nash char}
  \frac{1}{2}\int_0^T\E\|\beta^\ast \bar{\varphi}\|^2 dt +
  \E(y_0,\bar{\varphi}(0))
& = \sup_{\beta\in\cB}\left[\frac{1}{2}\int_0^T\E\|\beta \bar{\varphi}\|^2 dt +
  \E(y_0,\bar{\varphi}(0))\right],\notag\\
 \frac{1}{2}\int_0^T\E\|\beta^\ast \bar{\varphi}\|^2 dt +
  \E(y_0,\bar{\varphi}(0))
& = \inf_{\varphi\in\ol{X_\Theta}}\left[\frac{1}{2}\int_0^T\E\|\beta^\ast \varphi\|^2 dt +
  \E(y_0,\varphi(0))\right].
\end{align}
\end{theorem}

\appendix
\section{Appendix}
Let us recall some basics for the two-person zero-sum game
problem; for more details, see for example \cite[Chapter
8]{Aubin1993}.

There are two players: Emil and Francis. Consider a real-valued function $f:
E\times F\to \R$, where $f(x,y)$ is both the loss of Emil by taking
the strategy $x$ from her strategy set $E$ and the gain of Francis by
taking the strategy $y$ from his strategy set $F$ (the sum of the
gains is zero). Emil wants to minimize the function $f$, while Francis
wants to maximize $f$. The most important concept in the two-person
zero-sum game is the {\it Nash equilibrium}.
\begin{definition}
  Suppose that $E$ and $F$ are strategy sets of Emil and Francis,
  respectively. Let $f:E\times F\to \R$ be an index cost
  functional. We say $(\bar{x},\bar{y})\in E\times F$ is a Nash
  equilibrium, if
\[
f(\bar{x},y)\le f(\bar{x},\bar{y})\le f(x,\bar{y}),\ \forall x\in E,
y\in F.
\]
\end{definition}
The following result is well known, see, for instance, Proposition 8.1
in \cite[page 121]{Aubin1993}, which says seeking a Nash equilibrium
is equivalent to solving a minimax and a maxmini
problems, respectively, so that the extremes achieved are the same.
\begin{theorem}
  \label{thm:Nash}
  The pair $(\bar{x},\bar{y})$ is a Nash equilibrium if and only if 
\[
\sup_{y\in F} f(\bar{x},y) = V^+,\quad \inf_{x\in E}
f(x,\bar{y}) = V^-,
\]
and $V^+ = V^-$, where
\[
V^+:= \inf_{x\in E}\sup_{y\in F}f(x,y),\quad V^-:=\sup_{y\in
  F}\inf_{x\in E} f(x,y).
\]
\end{theorem}

\end{document}